\documentclass[12pt]{amsart}
\usepackage{amsmath,amsfonts,amssymb,amsthm, esint}
\usepackage[english]{babel}
\usepackage{enumerate}
\usepackage{dsfont}
\usepackage{comment}
\usepackage[colorlinks,citecolor=blue]{hyperref}
\usepackage{epsfig,graphicx}
\usepackage{latexsym, amsxtra, mathrsfs, bm}
\usepackage{graphics}
\usepackage{verbatim}
\usepackage[abs]{overpic}
\usepackage{mathtools}
\usepackage{amscd}
\allowdisplaybreaks[3]
\theoremstyle{plain}
        \newtheorem{theorem}{Theorem}[section]
        \newtheorem*{theorem*}{Theorem}
        \newtheorem*{conj*}{Conjecture}
        \newtheorem{lemma}[theorem]{Lemma}
        \newtheorem{prop}[theorem]{Proposition}
        \newtheorem{corollary}[theorem]{Corollary}

        \theoremstyle{definition}
        \newtheorem{definition}[theorem]{Definition}
        \newtheorem{rem}[theorem]{Remark}

\theoremstyle{remark}
        \newtheorem*{remark}{Remark}

        \newtheorem*{example}{Example}

\numberwithin{equation}{section}
\numberwithin{theorem}{section}
\numberwithin{table}{section}
\numberwithin{figure}{section}

\begin{document}
\title{Shift maps and statistical invariants for some dynamical systems}
\author{Sergey Kryzhevich}
\address[Sergey Kryzhevich]{
	Institute of Applied Mathematics, Faculty of Applied Physics and Mathematics, Gda\'nsk University of Technology, 80-233, Gda\'nsk, Poland}
\email[Sergey Kryzhevich]{sergey.kryzhevich@pg.edu.pl}
\author{Yiwei Zhang}
\address[Yiwei Zhang]{School of Mathematical Sciences and Big Data, Anhui University of Science and Technology, Huainan, Anhui 232001, P.R. China}
\email[Yiwei Zhang]{zhangyw@aust.edu.cn, yiweizhang831129@gmail.com}

\thanks{}

\begin{abstract} Given a dynamical system, we study the so-called space of shift functions thus introducing another vision on bifurcations and chaos. As an application of the obtained results, we give a partial solution to an open problem formulated in \cite{Misiurewicz1}: to describe all the one-dimensional maps with all the periodic orbits having the same mean value.

Moreover, we show that there are continuous families of such mappings having infinitely many periodic points. For this purpose, we study the dynamics of the so-called replicator maps, depending on two parameters. Such studies are also motivated by the analysis of the dynamics of evolutionary games under selection. We prove the existence of hyperbolic chaos for the considered map and demonstrate that the average values are the same for all the periodic orbits.
\end{abstract}


\maketitle

\section{Introduction}

We study the maps of the segment that have the fixed statistical invariants, for example, the mean value for all periodic trajectories. More precisely, we are interested in the following problem \cite[Problem 3.7]{Misiurewicz1}, see also the Abstract of the referred paper.

\noindent\textbf{Problem 1}.
\textit{
Describe all the continuous mappings
$T:{\mathbb R}\mapsto {\mathbb R}$ having periodic points of arbitrarily large period and such that the mean values of all the periodic points coincide.}

Let $\mathcal M$ be the class of the mappings satisfying the above property. Statistical properties of such mappings are discussed in \cite{art_stat}.

The first known example of such systems is given by the formula $x\mapsto Axe^x$. Later on, the replicator map    \begin{equation}\label{equ:Boltzmann0}
f(x)=f_{a,b}(x)=\frac{x}{x+(1-x) e^{a\left(x-b \right)}}
\end{equation}
was studied and it was noticed that all the periodic points of that map (except the fixed points $0$ and $1$) have the same average value $b$.
A detailed survey of the replicator maps theory and the related references are given separately in Section 5.

In this paper, we study a more general problem. Let $\mathfrak A$ be the Borel sigma-algebra engendered by the topology of a metric compact set $X$. Consider the space $E$ of all the bounded functions from $X$ to $\mathbb R$ measurable in $\mathfrak A$. Let $C=C(X\mapsto {\mathbb R})$ be the set of all the continuous functions on $X$.

In Sections 2 and 3, given a mapping $T$ of a compact set $X$, we consider the space
\begin{equation}\label{eqh1} H:=\{\psi\in C: \psi=\varphi-\varphi\circ T, \quad \varphi\in E\}
\end{equation}
of the so-called shift functions. For all these functions (and their limits in the space $C$) their mean values over any $T$-invariant measure equal zero. So, an enforced Problem 1 can be formulated as follows: to describe all the mappings of the segment with infinitely many ergodic invariant measures and  $x-b$ being in the closure of $H$ for some constant $b$. The results of Section 2 (Theorem 2.4) imply that the codimension of the closure $H$ in $C$ is infinite if and only if the first of the mentioned properties is satisfied.

In Section 4, we offer an explicit method to construct mappings with $x-b\in H$. Given a continuous invertible map $h:I\mapsto {\mathbb R}$ of an open connected subset $I\subset {\mathbb R}$ and constants
$a,b \in \mathbb R$, $a>0$ we demonstrate that the map
\begin{equation}\label{eqfh}
f_h:=h^{-1}\circ (h(x)+a(x-b))
\end{equation}
has the required property. In particular, taking
$h(x)=\ln \left(\frac{1-x}x\right)$, we get a replicator map. defined by Eq.\, \eqref{equ:Boltzmann0}.

Moreover, in Section 5, we study the properties of the replicator map (in particular, the bifurcations and hyperbolic chaotic invariant sets). All in all, this gives an opportunity to construct infinitely many maps of the class $\mathcal M$: it suffices to take in \eqref{eqfh} a map $h$ being $C^1$ close to $\ln \left(\frac{1-x}x\right)$.

The proof of Theorem 5.1 is given in the Appendix.

\section{Shift functions and ergodic theory}

Let $(X,d)$ be a compact metric space, $T:X\mapsto X$ be a continuous mapping, and $\Omega=\Omega_T$ be its non-wandering set. Consider the space of all the \emph{shift functions} defined by Eq.\eqref{eqh1}

Given a Banach space $K\supset H$, we consider the closure $H_K$ of $H$ in $K$. For example, $H_C$ is the closure of $H$ in $C$.

The space $H_C$ never coincides with $C$. Let $\mu$ be any Borel probability invariant measure for $T$. Given a function $\phi$ in
${\mathbb L}^1_\mu$ we consider it Birkhoff average $\hat \phi$.

\begin{lemma}\label{les1} For any invariant measure $\mu$ and any $\psi \in H_C$ we have $\hat\psi=0$.
\end{lemma}

\begin{proof} Note that for any Borel probability invariant measure $\mu$ we have
\begin{equation}\label{eqs1}
\int\limits_X (\varphi (x)-\varphi(Tx))\ d\mu=0.
\end{equation}
So, the integrals over invariant measures vanish on the space $H_C$. Then it suffices to apply the Birkhoff theorem.
\end{proof}

In particular, non-zero constants do not belong to $H_C$.

Next lemma demonstrates that, in general, the space $H_C$ is quite rich.

\begin{lemma} If $\varphi \in C$ is such that $\varphi|\Omega=0$ then $\varphi\in H_C$.
\end{lemma}

\begin{proof} We use the following folklore statement of the Topological Dynamics.

\begin{prop} Given a neighborhood $U$ of the non-wandering set $\Omega$ there exists a number $N$ such that any orbit of the map $T$ spends of at most $N$ steps outside of $U$.
\end{prop}

\begin{proof} Fix an open set $U\supset \Omega$. For any point $x_0 \in X\setminus U$, there is a neighborhood $V(x_0)$ such that $V(x_0)\cap T^n(V(x_0))=\emptyset$ for any $n\in {\mathbb N}$. The neighborhoods $V(x_0)$ cover the compact set $X\setminus U$; take a finite subcoverage
$\{V_1,\ldots,V_N\}$. Any trajectory cannot enter any of the sets $V_j$ more than once which proves the Proposition.
\end{proof}

The latter Proposition implies that for any function $\varphi$ supported on the completion of $U$, the series
$$\sum_{j=0}^\infty \varphi\circ T^j$$
converges (all the terms are bounded and, for every point $x$, at most $N$ of them do not vanish at $x$). Evidently, the corresponding limit function belongs to the space $E$. The functions vanishing on $\Omega$ can be approximated by those supported out of neighborhoods of $\Omega$ in the space of continuous functions.
\end{proof}

To continue, we need to prove the following statement. Given a function $\psi:X\mapsto {\mathbb R}$, we define $\psi_+=\max(\psi,0)$, $\psi_-=\max(-\psi,0)$, so $\psi=\psi_+-\psi_-$.

\begin{lemma}\label{lemm_inv}
Let $\mu$ be a Borel probability measure on $X$ (not necessarily invariant with respect to $T$.) Let $\psi\in {\mathbb L}^2_\mu$ be a function, orthogonal to all the functions of the space $H$. Then the measures $\nu_\pm$ such that $d\nu_\pm = \psi_{\pm} d\mu$  are invariant with respect to $T$ unless the corresponding functions $\psi_\pm$ vanish almost everywhere.
\end{lemma}

\begin{proof} Let
$$A_+=\{x:\psi(x)>0\}, \quad A_-=\{x:\psi(x)<0\},\quad A_0=\{x:\psi(x)=0\}.$$
Given a Borel set $B$, we denote the corresponding indicator function by ${\mathbf 1}_B$.

Due to the assumption of the lemma,
$$\int\limits_X \varphi \cdot \psi\, d\mu=\int\limits_X (\varphi\circ T) \cdot \psi\, d\mu$$
for any continuous function $\varphi$. So, the same is true for any bounded Borel function.

In particular,
$$\int\limits_X {\mathbf 1}_{A_+} \cdot \psi\, d\mu=\int\limits_X {\mathbf 1}_{T^{-1}(A_+)} \cdot \psi\, d\mu$$
This means that
\begin{equation}\label{eqmu0}
\int\limits_{A_+} \psi\, d\mu=\int\limits_{T^{-1}(A_+)} \psi\, d\mu=
\int\limits_{A_+} \psi\, d\mu-\int\limits_{A_+\setminus T^{-1}(A_+)} \psi\, d\mu+\int\limits_{T^{-1}(A_+)\setminus A_+} \psi\, d\mu
\end{equation}
The last two terms in the right hand side of the Eq.\eqref{eqmu0} are non-positive. Therefore
$$
\mu(A_+)\setminus T^{-1}(A_+)=\mu(A_-\cap T^{-1}(A_+))=0.
$$
and, consequently,
\begin{equation}\label{invariance}
\int\limits_{A_+} \eta\, \psi\, d\mu=\int\limits_{T^{-1}(A_+)} \eta\, \psi\, d\mu, \qquad \forall  \eta\in {\mathbb L}^2_\mu.
\end{equation}

Then, for any $\varphi\in C$
$$\begin{array}{c}
\int\limits_X \varphi d\nu_+=\int\limits_X \varphi \cdot \psi_+\, d\mu=
\int\limits_X \varphi\cdot {\mathbf 1}_{A_+}\cdot \psi\, d\mu=
\int\limits_X (\varphi\circ T)\cdot {\mathbf 1}_{T^{-1}(A_+)}\cdot \psi\, d\mu=\\[10pt]
[\mbox{due to \eqref{invariance}},\, \eta=\varphi\circ T]= \int\limits_X (\varphi\circ T)\cdot {\mathbf 1}_{(A_+)}\cdot \psi\, d\mu=\int\limits_X (\varphi\circ T) d\nu_+.
\end{array}$$
This proves that the measure $\nu_+$ is invariant if the function $\psi_+$ is non-zero.
\end{proof}

\begin{theorem}
Let $(X,\mu)$ be a probability space, and $T:X\mapsto X$ be a mapping, preserving the measure $\mu$. Let $L^2:={\mathbb L}^2_\mu$. Then $\mu$ is ergodic if and only if $L^2=H_{L^2}\oplus \langle 1 \rangle$. \end{theorem}

\begin{proof} Observe that due to \eqref{eqs1}, we have that any constant function is orthogonal to $H_{L^2}$ and hence, $1\notin H_{L^2}$.

First, suppose that the measure $\mu$ is not ergodic. Then there exists a measurable invariant set $A\subset X$ with $\mu(A)\in (0,1)$. Consider the conditional probabilities $\nu_1=\mu(\cdot|A)$ and $\nu_2=\mu(\cdot|X\setminus A)$. This measure is absolutely continuous with respect to $\mu$ with the density being a constant function on $A$ and 0 on its completion (and, hence, an $L^2$ function). Both these measures engender integral functionals that vanish on $H_{L^2}$ and take the value 1 on the constant 1. So, they coincide on $H_{L^2}\oplus \langle 1 \rangle$. On the other hand, they must differ on $L^2$, so $H_{L^2}\oplus \langle 1 \rangle\neq L^2$.

Now let the measure $\mu$ be ergodic. Suppose that $H_{L^2}\oplus \langle 1 \rangle\neq L^2$. The space $H_{L^2}\oplus \langle 1 \rangle$ is closed in $L^2$, so we can take a function $\psi\in L^2$,
$\psi \neq 0$, orthogonal to the space $H_{L^2}\oplus \langle 1 \rangle$. We have
$$\int\limits_X \psi \, d\mu=0,$$
so the functions $\psi_+$ and $\psi_-$ are both nonzero. Due to Lemma \ref{lemm_inv}, these functions define mutually singular invariant measures both being absolutely continuous with respect to $\mu$. This contradicts to the assumption that $\mu$ is ergodic.
\end{proof}

The following statement illustrates the relation between invariant measures and the space $H$.
Let $m$ be the dimension of the set of all the invariant measures and $n$ be the codimension of the space $H_C$ in $C.$

\begin{theorem} If $n$ is finite then $m$ is finite and vice-versa. Moreover, in this case, $m=n$.
\end{theorem}

\begin{proof}

\noindent\textbf{1.} Suppose that $n$ is finite, and $m>n$ (including the case $m=\infty$). Let $\varphi_1=1, \varphi_2,\varphi_n$ be continuous functions such that
$$C=H_C\oplus \langle \varphi_1,\ldots, \varphi_n\rangle.$$
Suppose that there exist $n+1$ distinct ergodic invariant measures $\mu_1,\ldots,\mu_{n+1}$. Then there exist constants $c_1,\ldots,c_{n+1}$ (not all zeros) such that
$$\sum_{j=1}^{n+1} c_j \int\limits_X \varphi_i \, d\mu_j =0, \qquad i=1,\ldots, n.$$
Without loss of generality, one may assume that there exists a number
$k\in \{1,\ldots,n\}$ such that
$$c_1,\ldots c_k\ge 0, \qquad c_{k+1},\ldots,c_{n+1}<0.$$
Then
$$\nu_1:=c_1\mu_1+\ldots c_k\mu_k=-c_{k+1}\mu_{k+1}-\ldots -c_{n+1}\mu_{n+1}=:\nu_2$$
since they coincide as operators on the space $C$. However, all the measures $\mu_j$ are mutually singular and the same must be true for measures $\nu_1$ and $\nu_2$. This contradiction demonstrates that if $n$ is finite, then $m\le n$.

\noindent\textbf{2.} Suppose that the set of invariant measures has the dimension $m$ or, in other words, there exist $m$ ergodic invariant measures $\mu_1, ..., \mu_m$ and all other invariant measures are their convex combinations. If the codimension $n$ of $H_C$ in $C$ is greater than $m$, there must be a function $\varphi\notin H_C$ such that
\begin{equation}\label{eqintphimu}
\int\limits_X \varphi \, d\mu_j=0,\qquad j=1,\ldots, m
\end{equation}
Define the space $H_1:=H_C\oplus \langle \varphi\rangle$ (this is a closed subset of $C$) and define a linear functional $J$ on $H_1$ such that $J|_{H_C}=0$ and $J\varphi =1$. Evidently, the functional $J$ is continuous on $H_1$ and, due to the Hahn-Banach theorem, it can be extended to a continuous linear functional on $C$. By the Riesz representation theorem, there is a bounded signed Borel measure $\nu$ such that
$$J\psi = \int\limits_X \psi \, d \nu$$
for any $\psi\in C$.
There may be three options.
\begin{enumerate}
\item The signed measure $\nu$ can be represented as a difference
$\nu_+-\nu_-$ of two mutually singular finite Borel measures.
\item $\nu=\nu_+$ is a finite Borel measure.
\item $\nu=-\nu_-$ where $\nu_-$ is a finite Borel measure.
\end{enumerate}

We consider, without loss of generality, the first option. Proceeding, if necessary, to a renormalization of the signed measure $\nu$, we may assume that $\nu_+$ is a probability measure, and $\int\limits_X \varphi \,d\nu_+\neq 0$. This is because $J\psi \neq 0$, so, at least, one of integrals $\int\limits_X \varphi \,d\nu_\pm$ is non-zero.

On the other hand, $J$ vanishes on $H_C$, so we have
$$\int\limits_X \psi \,d \nu=\int\limits_X \psi\circ T\, d\nu, \qquad \forall \psi\in C
$$
and a similar property is true for the component $\nu_+$, see Lemma \ref{lemm_inv}. Therefore, the latter measure is invariant. Meanwhile,
$$1=J\varphi=\int\limits_X\psi\, d\nu =\int\limits_X\psi\, d\nu_+-\int\limits_X\psi\, d\nu_-.$$
So,
$$\int\limits_X\psi\, d\nu_+>0$$
and, due to \eqref{eqintphimu}, the measure $\nu_+$ cannot be represented as a convex combination of measures $\mu_1,\ldots,\mu_m$, a contradiction. Therefore, $n\le m$.
\end{proof}

\begin{corollary} The mapping $T$ is uniquely ergodic if and only if
$$C=H_C\oplus \langle 1 \rangle.$$
\end{corollary}

\begin{example} For an irrational rotation of a circle, any function with a zero average over the circle is of the class $H_C$. The proof is based on the Fourier decomposition (similarly to the proof of the Kronecker - Weyl Theorem).
\end{example}

\begin{example} Let $T:S^1\mapsto S^1$ be the circle doubling map defined by the formula $Tz=z^2$. The mapping $T$ admits infinitely many periodic points, so the codimension of $H_C$ in C is infinite. Meanwhile, $T$ preserves the Lebesgue measure $\mathbf{m}$ on the circle which is ergodic, so the closure of $H$ in $\mathbb{L}^2_{\mathbf{m}}$ has codimension 1 (in fact, this is the space of all the functions with zero average).
\end{example}

\section{Bifurcations, chaotic invariant sets and the Misiurewicz problem.}

The techniques developed in the previous section, allows for a different perspective on bifurcations and chaos. Later on, we study smooth mappings on segments, line or smooth manifolds. Let $T:M\mapsto M$ be a smooth map where $M$ is one of the above sets.

\begin{definition} Let $X$ be a compact invariant subset of the mapping $T$. We say that the set $X$ is \emph{$H$-stable} if there exists an $\varepsilon>0$ and a neighbourhood $U$ of the set $X$ such that for any mapping $S:M\mapsto M$, which is $\varepsilon$ - close to $T$ in the $C^1$ topology at $U$, the neighbourhood $U$ contains a maximal invariant set $Y$. Moreover, there exists an isomorphism of the spaces of continuous functions $C(X)$ and $C(Y)$ which maps $H_C(X,T)$ onto $H_C(Y,S)$. Note that, due to the Stone-Banach theorem, the sets $X$ and $Y$ must be homeomorphic.

If $X$ is not $H$-stable, we say that an $H$ -- bifurcation is observed.
\end{definition}

Observe that the classical one-dimensional bifurcations (saddle-node, pitchfork, period-doubling and Hopf bifurcation) can be regarded as local $H$ - bifurcations.

\begin{definition} We say that the compact invariant set $X$  of the map $T$ is $H$-chaotic if there exists a Borel probability invariant measure $\mu$ supported on the set $K=\mathrm{supp}\, \mu$ such that
\begin{enumerate}
\item $H_{{\mathbb L}^2_\mu}\oplus \langle 1\rangle={\mathbb L}^2_\mu$;
\item The codimension of the $H_{C(K)}$ in $C(K)$ is infinite.
\end{enumerate}
\end{definition}

This means that the measure $\mu$ is ergodic but there are infinitely many ergodic invariant measures supported on subsets of $K$.
This definition is satisfied for the classical examples of chaos like hyperbolic invariant sets, satisfying Devaney's definition.

And now we demonstrate how the Problem 1 formulated in the Introduction can be partially resolved by using the shift functions techniques.

The following statement is obvious.

\begin{theorem}
Let $X$ be the closure of the periodic points of the map $T$. Suppose that the map admits the orbits of arbitrarily large periods.  Let $x-a\in H_C$ for some $a\in {\mathbb R}$. Then $T\in {\mathcal M}$ and $a$ is the mean value of those periodic orbits.
\end{theorem}

In the end of this section, we provide a simple result which allows to check if a given function belongs to the space $H$ (also, the lemma below demonstrates that the space $H$ may be reacher than the space that of functions $g$ for which the equation $f\circ T-f=g$ is solvable).

Let $T\in C(X\mapsto X)$ be a minimal uniquely ergodic topological dynamical system with the Borel probability invariant measure $\mu$.

\begin{definition}
For a function $f\in C(X)$ define the coboundary of $f$ by
$$\mbox{co}\,(f) = \sup\limits_{n\in {\mathbb N},x \in X}
\left| S_n(f)(x) - n\int\limits_X f\, d\mu\right|$$
where
$$S_n(f):= \sum_{i=0}^n f \circ T^i.$$
\end{definition}

The following equivalence property is from statements (1), (3), (4) in  \cite[Theorem 14.11]{gotthed:55:topdyn}

\begin{lemma} If $(X, T)$ is a minimal topological dynamical system and $g\in C(X)$, then the following statements
are equivalent:
\begin{enumerate}
\item there exists some $f\in C(X)$, such that $f\circ T - f = g$;
\item there exists some $x_0 \in X$, such that the sequence
$$\{S_n(g)(x_0)\}_{n\in {\mathbb N}}$$ is bounded.
\item $\{S_n(g)\}_{n\in {\mathbb N}}$ is uniformly bounded in $X$, i.e., $$\sup_{n\in {\mathbb N},x\in X}|S_n(g)(x)| < +\infty.$$
\item The coboundary of $\mathrm{co}\, g$ of the function $g$ is finite.
\end{enumerate}
\end{lemma}

\section{One-dimensional maps with fixed mean values of periodic orbits}

Define a specific class of one-dimensional systems.

\begin{definition}\label{dm}
Let $I$ be a convex subset of $\mathbb R$. We say that a measurable mapping $f:I \mapsto I$ is of the class \emph{${\mathcal M}'$} if there exists a constant $b\in {\mathbb R}$ such that for any Borel probability $f$-invariant measure $\mu$ we have
\begin{equation}\label{eqb}
\int\limits_{I} x \, d\mu =b.
\end{equation}
In particular, for any $n$-periodic point $x_0$ (i.e. $f^n(x_0)=x_0$), we have
\begin{equation}\label{eqb1}
\dfrac1{n} \sum_{j=0}^{n-1} f^j(x_0) =b.
\end{equation}
If we need to specify the constant $b$ in conditions \eqref{eqb} and \eqref{eqb1}, we say that the mapping $f\in {\mathcal M}'_b$.
\end{definition}

In \cite[Corollary 3.5]{Misiurewicz1} it was
proven that $f_{a,b}\in {\mathcal M}'_b$
where $f_{a,b}$ is the replicator map considered in following sections.
This solves a known question from \cite{Misiurewicz3}: whether there are any smooth maps of ${\mathcal M}'$, with infinitely many periodic points other than
\begin{equation}\label{aeminusx}
x\mapsto Axe^{-x}.
\end{equation}
Besides, an open question was formulated: to find all smooth maps having infinitely many periodic orbits for which the centers of mass of all periodic orbits coincide.

We are still quite far from being sure to construct all the possible maps of ${\mathcal M}'$. However, in this section, we describe a very broad class of such mappings.

Let us start with a trivial observation.

\begin{prop}
Let $f:{\mathbb R} \to {\mathbb R}$ be a continuous mapping of a convex subset of a line. If there is a piecewise continuous function $h: I \mapsto {\mathbb R}$ and constant $a,b\in {\mathbb R}$, $a>0$ such that
\begin{equation}\label{eqh}
h(f(x))-h(x)=a(x-b)
\end{equation}
then $f\in {\mathcal M}_b$.
\end{prop}

In fact, we can always take $a=1$ rescaling the function $h$.

Equation \eqref{eqh} implies that for any $n\in {\mathbb N}$ we have
$$h(f^n(x))-h(x)=a(x+f(x)+\ldots+f^{n-1}(x))-nab.$$

The latter equation implies the equality
$$\lim_{n\to \infty} \dfrac{1}{n} (x+f(x)+\ldots+f^{n-1}(x))=b$$
for all $x$ such that $h(f^n(x))$ is bounded (or, at least, $h(f^n(x))=o(n)$).

\begin{rem}
    Observe that if the mapping $h$ is invertible, equation \eqref{eqh} implies that
    any map
    \begin{equation}\label{eqgg}
    f(x)=h^{-1}(h(x)+a(x-b))\in {\mathcal M}_b.
    \end{equation}
    If
    $$h(x)=\ln\left(\dfrac{1-x}{x}\right),$$ the formula \eqref{eqgg} gives us the mapping $f_{a,b}$ introduced above; for $h(x)=-\ln x$, we get the mapping \eqref{aeminusx} with $A=e^b$. However, equation \eqref{eqgg} can generate multiple examples of mappings from ${\mathcal M}_b$ (also having some other similar properties). For example, one can consider the family $$\arctan (\tan x - a(x-b))$$
    or take $h=\Phi^{-1}$ where $\Phi$ is the cumulative function of the normal distribution.
\end{rem}

\begin{rem}
Theorem \ref{Theorem1}, given below, guarantees that the replicator map $f_{a,b}$ has a hyperbolic chaotic invariant set for some values  $a=a_0$, $b=b_0$. Then we can take
$${\tilde f}_h=h^{-1}(h(x)+a_0(x-b_0))$$
for some $h$, $C^1$ close to $\ln((1-x)/x)$ in the absorbing segment of $f_{a_0,b_0}$ (actually, this is $[f_{max},f_{min}]$). We get a map of the class $\mathcal M$ that still has a chaotic invariant set and, consequently, infinitely many periodic points.
\end{rem}

\section{Basic properties of replicator maps}
We consider the following so-called replicator maps defined by formula  
\begin{equation}\label{equ:Boltzmann}
f_{C,V,\gamma}(x)=\frac{x}{x+(1-x) e^{\frac{C\gamma}2 \left(x-\frac{V}{C}\right)}}
\end{equation}

In this section, we will discuss some basic properties of the family of replicator maps \eqref{equ:Boltzmann}. More details and references about applications of such maps (and, also, of conjugated maps \eqref{eqeos1}, see below) may be found in, \cite{Misiurewicz3}, \cite{EOS96}, cite{aks} and, also, in the preprint \cite{chmara_et_al}.

\subsection{Reparametrization}
It is easy to see that, in fact, the replicator map depends on only two parameters: $a:=c\gamma/2$ and $b:=v/c$. So, we rewrite \eqref{equ:Boltzmann} in the form \eqref{equ:Boltzmann0}, given in the Introduction
with parameters $a>0$ and $b\in (0,1)$. See Figure \ref{fig_1} for the sample graph of $f_{a,b}$.

\begin{figure}[t!]\centering
  \includegraphics[width=.5\textwidth]{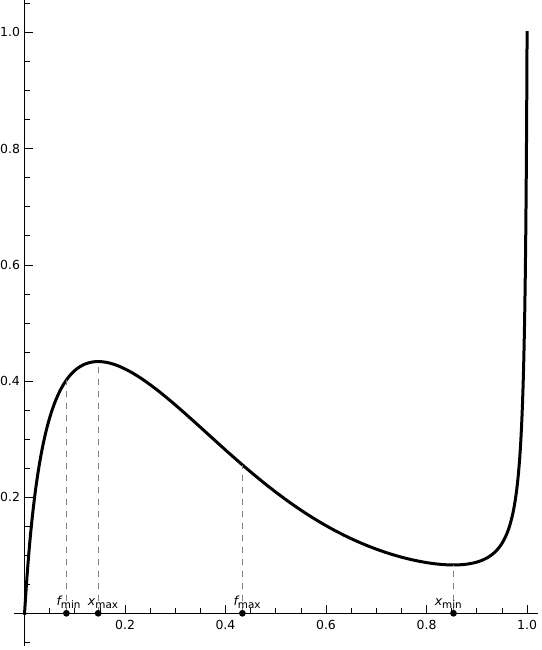}
\caption{\footnotesize The map $f_{8,1/3}$.}
\label{fig_1}
\end{figure}

\subsection{Fixed points}

Observe that the map $f$ has exactly three fixed points $0$, $b$, and $1$. To study their stability, let us calculate
\begin{equation}\label{eq0}
f'_{a,b}(x)=\dfrac{e^{a\left(x-b\right)}\left(a x^2-a x+1\right)
}{\left(x+(1-x)e^{a\left(x-b\right)}\right)^2}.
\end{equation}
Consequently,
$$f'_{a,b}(0)=e^{a b}>1, \qquad f'_{a,b}(1)=e^{a (1-b)}>1.$$
Evidently, both points $0$ and $1$ are unstable.

Meanwhile
$$f'_{a,b}\left(b\right)=1-a b(1-b)<1$$
Consequently, the point $b$ is stable if $f'(b)>-1$, that is
$$a<\dfrac{2}{b(1-b)},$$
and unstable if
$$a>\dfrac{2}{b(1-b)}.$$

In addition, from Equation \eqref{eq0} it follows that the map $f_{a,b}$ is monotone if $a\le 4$ (so its dynamics is trivial). In what follows, we always assume that \begin{equation}\label{eq1}
a>4.
\end{equation}

Moreover, in \cite[Theorem 3.8]{Misiurewicz1} it was shown that the fixed point $b$ attracts all the points of the interval $(0,1)$ provided in $a\le 2/b(1-b)$.

\subsection{Recursive formula}

It follows from \cite[Equation (13)]{Misiurewicz1} that the $n$-th iteration of our map $f_{a,b}$ is given by the formula:
$$f_{a,b}^n(x)=\frac{x}{x+(1-x) e^{a\left(\sum_{i=0}^{n-1}\left(f_{a,b}^i(x)-b\right)\right)}},\quad \forall x \in[0,1], \quad \forall n \in {\mathbb N}.$$

\subsection{Critical points}

Let us observe, first of all, that  if \eqref{eq1} is true, the mapping
$f_{a,b}$ has exactly two critical points:
$$
x_{max} = \dfrac12-\sqrt{\dfrac14-\dfrac1{a}}
\quad\mbox{and}\quad
x_{min} = \dfrac12+\sqrt{\dfrac14-\dfrac1{a}}.
$$
Evidently, $x_{min}+x_{max}=1$ and $x_{min} x_{max}=1/a$. This gives us the following asymptotic estimates for those values:
$$
x_{max}=\dfrac1{a}+o\left(\dfrac1{a}\right), \qquad
x_{min}=1-\dfrac1{a}+o\left(\dfrac1{a}\right).
$$
Besides, we can estimate
$$
f_{min}:=f_{a,b}(x_{min})=\dfrac{1}{1+(1/x_{min}-1)e^{a(x_{min}-b)}}=e^{a(b-1)+o(a)},
$$
$$
f_{max}:=f_{a,b}(x_{max})=\dfrac{1}{1+(1/x_{max}-1)e^{a(x_{max}-b)}}=1-e^{-ab+o(a)}
$$
as $a\to +\infty$.

\subsection{Symmetry}

\begin{lemma}\label{lemma1} \cite[Equation (10)]{Misiurewicz1}.
For any $x,b\in [0,1]$ and $a>0$, we have
$$f_{a,1-b}(x)=1-f_{a,b}(1-x).$$
\end{lemma}

This means that the transformation $x\mapsto 1-x$ converts the map $f_{a,b}$ to $f_{a,1-b}$.
In particular,
$$f^n_{a,1-b}(x)=1-f^n_{a,b}(1-x)$$
for any $n\in {\mathbb N}$. So, it suffices to study the case $b\le 1/2$ only. Moreover, the case $b=1/2$ is quite trivial: there are no periodic solutions of periods higher than 2 (see \cite{Misiurewicz1}).

\subsection{Schwartzian derivative}

Recall the notion of the Schwartzian derivative:
$$Sf_{a,b}(x)=\left(\dfrac{f''_a(x)}{f'_{a,b}(x)}\right)'-\dfrac12
\left(\dfrac{f''_a(x)}{f'_{a,b}(x)}\right)^2.$$

The following technical statement is proven in \cite[Proposition 3.2]{Misiurewicz1}.

\begin{lemma}\label{Lemma5} For any $a>4$, any $b\in (0,1)$ and any $x\in (0,1)\setminus \{x_{min},x_{max}\}$ we have $Sf_{a,b}(x)<0$ while
$$Sf_{a,b}(x_{min})=Sf_{a,b}(x_{max})=-\infty.$$
\end{lemma}

\begin{corollary}
For any $a>4$ and any $b\in (0,1)$, the mapping $f_{a,b}$ has at most 2 minimal attracting sets.
\end{corollary}

\begin{proof} This corollary follows from \cite[Theorem 11.4]{Devaney}. Indeed, in the considered case (a mapping on a compact interval), any of those sets attracts at least one critical point.
\end{proof}

\subsection*{Transformation of coordinates}\label{subsec:conjugacy}
Introduce a family of maps $g_{a,b}:\mathbb{R}\to \mathbb{R}$ by
\begin{equation}\label{eqeos1}
g_{a,b}(y):= y + \dfrac{a}{e^y+1}-ab. 
\end{equation}
with $a>0$ and $b\in(0,1)$. See Figure \ref{fig_2} for the graph of $g_{30,1/3}$. 

\begin{figure}[t!]\centering
  \includegraphics[width=.5\textwidth]{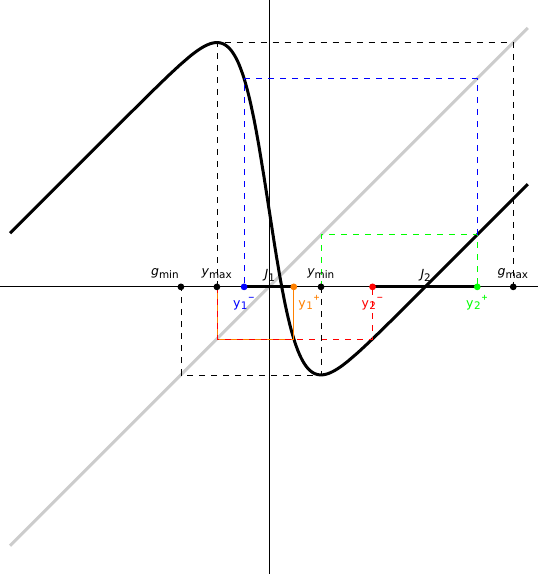}
\caption{\footnotesize The graph of the map $g_{30,1/3}$.}
\label{fig_2}
\end{figure}

This family of maps $\{g_{a,b}\}$ has been introduced in \cite{EOS96},
and has already been studied, see \cite{aks} and references therein. It was shown in \cite{Misiurewicz4} that for each $a,b$ by taking
$$y=h(x)=\ln\frac{1-x}{x}, \qquad
\left( x=\dfrac1{e^y+1} \right)$$
we have
\begin{equation}\label{eqeos2}
g_{a,b}(y)=h\circ f_{a,b}\circ h^{-1}(y). 
\end{equation}
In other words, the following diagram commutes
\[
\begin{CD}
(0,1) @>f_{a,b}(x)>> (0,1) \\
@VhVV @VhVV \\
\mathbb{R} @>g_{a,b}(y)>> \mathbb{R}
\end{CD}
\]
Therefore, $f_{a,b}$ and $g_{a,b}$ are smoothly conjugated.

Consequently, the results of the papers \cite{Misiurewicz1} and \cite{aks} can be merged. For example, one can classify the absorbing intervals of the mapping $f_{a,b}$ as it was done in \cite{aks} for $g_{a,b}$.

\section{Chaotic dynamics, hyperbolic chaos}

The main claim of this section states $f_{a,b}$ can be symbolically represented as a subshift of the Bernoulli shift of finite type (and this guarantees the existence of periodic points of any periods) in the case $b\neq 1/2$ (for $b=1/2$ there is no orbit of period greater than 2, see \cite{Misiurewicz1}).

Let $\Sigma$ be the set of all one-sided infinite sequences of symbols '0' and '1' endowed with the standard metrics:
$$\Sigma=\{\sigma=\left(\sigma_j\in \{0,1\}:j\in {\mathbb N}\cup\{0\} \right)\};$$
$$d_\Sigma(\sigma,\theta)=\sum_{j=0}^\infty \dfrac{|\sigma_j-\theta_j|}{2^j}.$$
Consider a subset $\Sigma_0\subset \Sigma$ given by the following formula:
$$\Sigma_0=\{\sigma\in \Sigma: \sigma_j\sigma_{j+1}=0,\quad \forall j\in {\mathbb N}\cup\{0\}\}.$$
In other words, the set $\Sigma_0$ consists of segments without neighboring symbols '1'.

We list some evident properties of $\Sigma_0$.
\begin{enumerate}
\item $\Sigma_0$ is an infinite closed subset of $\Sigma$ and, hence, compact. This is because the set of all sequences of $\Sigma$ with '1' at positions $j$ and $j+1$ is open for any $j$;
\item $\Sigma_0$ is invariant with respect to the standard shift mapping $S$ on $\Sigma$ (erasing the first symbol of a sequence).
\item The periodic points of the map $S$ are dense in $\Sigma_0$. This is because the periodic points may be obtained by an infinite repetition of finite sequences of digits (if this does not result in entries '11'), e.g. $010\,010 \ldots$ or $10100\,10100 \ldots$.
\item For any $m\in {\mathbb N}$ the set $\Sigma_0$ contains at least one point of period $m$ (for instance, the point corresponding to the repeating sequence $10\ldots 0$ for $m>1$ and the sequence of zeros if $m=1$).
\item The set $\Sigma_0$ is transitive with respect to $S$.
The point with a dense orbit may be obtained by writing down all the finite admissible sequences of '0' and '1', separated by '0':
$$\sigma^*= 0\, 0\, 1\, 0\, 00\, 0\, 01\, 0\, 10\, 0\, 000\, 0\, 001\,0 \, 010\, 0\, 100\, 0\, 101\, 0 \ldots$$
\end{enumerate}

\begin{theorem}\label{Theorem1} For any $b\in (0,1)$, $b\neq 1/2$, there is $a_0=a_0(b)>0$ such that for any $a>a_0$ there exists the compact set $Q\subset [f_{min},f_{max}]$, invariant with respect to $f_{a,b}$ and such that $f_{a,b}|_Q$ is topologically conjugated to the shift of finite type on the set $\Sigma_0$.
\end{theorem}

\section*{Appendix A. Proof of Theorem \ref{Theorem1}}

\subsection*{Key lemma for $g_{a,b}$}
In the view of \eqref{eqeos2}, to prove Theorem \ref{Theorem1}, it suffices to prove the following lemma for the function $g_{a,b}$ defined by \eqref{eqeos1}.

\noindent\textbf{Lemma A.1} (Key Lemma).
\emph{For any $b\in (0,1)$, $b\neq 1/2$, there is $a_0=a_0(b)>0$ such that for any $a\geq a_0$ there exists the compact set $K\subset {\mathbb R}$, invariant with respect to $g_{a,b}$ and such that
\begin{enumerate}
\item $|(g^2_{a,b})'(x)|>1$ for any $x\in K$;
\item the mapping $g_{a,b}|_K$ is topologically conjugated to the Bernoulli shift of the set $\Sigma_0$.
\end{enumerate}}

\begin{proof}
We assume, without loss of generality, that $b<1/2$, we always suppose this value to be fixed.

The map $g_{a,b}$ has the unique fixed point $y_0=\ln((1-b)/b)$.

The derivative of $g_{a,b}$ is given by the formula
$$g'_{a,b}(y)=1-\dfrac{ae^y}{(e^y+1)^2}.$$
The critical points are
$$y_{max}=\ln \left(\dfrac{a}2-1-\sqrt{\dfrac{a^2}4-a}\right)$$
and
$$y_{min}=\ln \left(\dfrac{a}2-1+\sqrt{\dfrac{a^2}4-a}\right).$$
Observe that $y_{min}+y_{max}=0$ and
$$e^{y_{min}}=\dfrac{a}2-1+\sqrt{\dfrac{a^2}4-a}=a+o(a)$$
as $a\to+\infty$.
The function $g_{a,b}$ increases on $(-\infty,y_{max}]$ and $[y_{min},+\infty)$ and decreases on $[y_{max},y_{min}]$.

In our assumptions, the following inequalities hold:
$$y_{max}<0<y_0<y_{min}.$$

The second derivative of $g_{a,b}$ may be calculated by the formula:
$$
g''_{a,b}(y)=\dfrac{ae^y(e^y-1)}{(e^y+1)^3}.
\leqno (A.1)$$
Evidently, this function is positive for positive values of $y$ and vice versa. The shape of $g_{a,b}$ is as illustrated in Figure \ref{fig_2}.

Denote
$$
g_{min}:=g_{a,b}(y_{min})\qquad g_{max}:=g_{a,b}(y_{max}),
$$
then the following statement holds true.

\noindent\textbf{Lemma A.2.} \emph{The value $a_0$ can be selected sufficiently large, so that
$$
g_{\min}<y_{\max}, \qquad g_{\max}>y_{\min},\qquad g_{a,b}
\leqno (A.2)$$
for any $a\ge a_0$.}

\begin{proof}[Proof of Lemma A.2.] Let us verify Lemma A.2. item by item. The first inequality of (A.2) can be rewritten as follows
$$
\ln\left(\dfrac{a}2-1+\sqrt{\dfrac{a^2}4-a}\right)+
\dfrac{a}{\frac{a}2+\sqrt{\frac{a^2}4-a}}-ab<\ln\left(\dfrac{a}2-1-\sqrt{\dfrac{a^2}4-a}\right)
$$
or
$$
2\ln\left(\dfrac{a}2-1+\sqrt{\dfrac{a^2}4-a}\right)+
\dfrac{a}{\frac{a}2+\sqrt{\frac{a^2}4-a}}-ab<0
\leqno (A.3)$$
Here we took into account the fact that $y_{min}+y_{max}=0$. The first term of the left-hand side (A.3) equals $2\ln a + o(1)$, the second one is $1+o(1)$, and the third one is $-ab$. All the expression is negative for sufficiently big $a$ provided $b$ is positive.

The second inequality of (A.2) looks as follows:
$$\ln\left(\dfrac{a}2-1-\sqrt{\dfrac{a^2}4-a}\right)+
\dfrac{a}{\frac{a}2-\sqrt{\frac{a^2}4-a}}-ab>\ln\left(\dfrac{a}2-1+\sqrt{\dfrac{a^2}4-a}\right)
$$
$$
\dfrac{a}{\frac{a}2-\sqrt{\frac{a^2}4-a}}-ab>2\ln\left(\dfrac{a}2-1+\sqrt{\dfrac{a^2}4-a}\right)
\leqno (A.4)$$
that is $a-ab+o(a)>2\ln (a+o(a))$. Here we took into account the fact that
$$\dfrac{a}{\frac{a}2-\sqrt{\frac{a^2}4-a}}=\frac{a}2+\sqrt{\frac{a^2}4-a}=a+o(a).$$
The inequality (A.4) is obviously true for big values of $a$ since $b<1$.

Finally, the last inequalities of (A.2) can be expressed in the form
$$g_{max}+\dfrac{a}{e^{g_{max}}+1}-ab>y_{min}$$
or
$$y_{max}+\dfrac{a}{e^{y_{max}}+1}+\dfrac{a}{e^{g_{max}}+1}-2ab>y_{min}$$
or
$$2y_{max}+\dfrac{a}{e^{y_{max}}+1}+\dfrac{a}{e^{g_{max}}+1}-2ab>0.$$

Substituting the exact values of $y_{max}$ and $g_{max}$, we obtain
$$\begin{array}{c}
2 \ln \left({\dfrac{a}2-1-\sqrt{\dfrac{a^2}4-a}}\right)+\dfrac{a}{\dfrac{a}2-\sqrt{\dfrac{a^2}4-a}}+\\
\dfrac{a}{1+\left(\dfrac{a}{2}-1-\sqrt{\dfrac{a^2}{4}-a}\right) \exp\left(\dfrac{a}{a/2-\sqrt{a^2/4-a}}-ab\right)}-2ab>0.
\end{array}$$
The first summand on the last formula is $o(a)$, the second is $a+o(a)$, the third one tends to 0 and the last one is $-2ab$. Therefore, the overall sum is $a(1-2b)+o(a)$ which must be positive. We complete the proof of the lemma.
\end{proof}

\medskip

As it follows from the statement of the Lemma, $g_{min}<y_{max}<y_{min}$ and the map $g_{a,b}$  increases on $(y_{min},+\infty)$.
Consequently (as well as plotted in Figure \ref{fig_2}), there is a unique point $y_2^+$ in $(y_{min},+\infty)$ such that $g_{a,b}(y_2^+)=y_{min}$.

Now, we prove that $y_2^+\in g_{a,b}[y_{max},y_{min}]$. Evidently, $g_{min}=g_{a,b}(y_{min})<y_{min}<y_2^+$.

Now let us prove that $y_2^+\le g_{max}$. We take into account the inequality $g_{max}>y_{min}$ and the fact that $g_{a,b}$ increases for $y>y_{min}$. So, if there were $g_{max} < y_2^+$, we would have $g_{a,b}(g_{max})< g_{a,b}(y_2^+)=y_{min}$, which contradicts the last of the inequalities (A.2).

Then there exists a unique point $y_1^-\in [y_{max},y_{min}]$ such that $g_{a,b}(y_1^-)=y_2^+$.

Besides, $y_{max}\in [g_{min},g_{max}]$ for big values of $a$. So, there are unique points $y_1^+\in [y_{max},y_{min}]$ and $y_2^-\in [y_{min},+\infty)$ such that
$$g_{a,b}(y_1^+)=g_{a,b}(y_2^-)=y_{max}.$$

Since $g_{a,b}$ is decreasing on $[y_{\max},y_{\min}]$ and increasing on $[y_{\min},+\infty)$, we have $y_{1}^{-}<y_{1}^{+}$ and $y_{2}^{-}<y_{2}^{+}$. Therefore, we are able to consider two segments:
$$
J_1:=[y_1^-,y_1^+]\subset (y_{max},y_{min})~~~ \mbox{and}~~~~
J_2:=[y_2^-,y_2^+]\subset (y_{min},+\infty).
$$

These two segments are disjoint. Moreover, we have
$$
g_{a,b}(J_1)=[y_{\max},y_2^+]\supset J_1\cup J_2, ~~~  g_{a,b}(J_2)=[y_{\max},y_{min}]\supset J_1, ~~~ \mbox{and}~~~
g_{a,b}(J_2)\cap J_2=\emptyset$$
These facts immediately imply that
$$g_{a,b}^2(J_i)\supset J_1\cup J_2, \qquad i=1,2.$$
Applying $g_{a,b}$ to the latter inclusion, we can easily see that
$$g_{a,b}^m(J_i)\supset J_1\cup J_2, \qquad i=1,2$$
for any $m\ge 2$.
We set
$$
K:=\bigcap_{k=0}^{\infty} g_{a,b}^{-k} (J_1\cup J_2).
\leqno (A.5)$$

This $K$ in (A.5) is a nonempty compact invariant set. Given a point $y\in K$, we construct a sequence $\eta(y)=(\sigma_j: j\in {\mathbb N}\cup\{0\})\in \Sigma$ as follows:
$\sigma_m=\epsilon\in \{0,1\}$ if $g^m_{a,b}(y)\in J_{\epsilon+1}$.

Since $g_{a,b}(J_2)\cap J_2 =\emptyset$, we have
$\eta(y)\in \Sigma_0$ for any $y\in K$.

Now we prove that for any point of $\Sigma_0$ there exists a corresponding point of $K$.

\noindent\textbf{Lemma A.3.} \emph{Given a sequence $\Bar{\sigma}\in \Sigma_0$, there exists a point $\Bar{y}\in \Sigma_0$ such that
$$\eta(\Bar{y})=\Bar{\sigma}.$$}

\begin{proof} Let $\Bar{\sigma}=(\sigma_j:j\in {\mathbb N}\cup\{0\})$. We construct a nested sequence of segments
$$I_0\supset I_1 \supset I_2\supset \ldots$$
corresponding to $\Bar{\sigma}$. First of all, we take $I_0=J_{\sigma_0+1}$.

Let the segment $I_m$ be already constructed so that
$g_{a,b}^m(I_m)=J_{\sigma_{m}+1}$. Since $\Bar{\sigma}\in \Sigma_0$, we have $g_{a,b}(J_{\sigma_m+1})\supset J_{\sigma_{m+1}+1}$. So, we can take a subsegment $I_{m+1}\subset I_m$ so that $g_{a,b}^m(I_{m+1})=J_{\sigma_{m+1}+1}$.

We repeat this procedure infinitely many times. So, we can take any point ${\Bar{y}}\in \bigcap_{m=0}^\infty I_m$.
\end{proof}

To prove that the mapping $\eta$ is one-to-one, we demonstrate that the mapping $g_{a,b}$ is uniformly expanding on each of the sequences $J_i$.
So, it suffices to estimate the derivatives of $g_{a,b}^2$.

\noindent\textbf{Lemma A.4.} \emph{The value $a_0$ can be taken so big that
$$
\min\left\{|g'_{a,b}(y_1^-)|,|g'_{a,b}(y_1^+)|\right\}\cdot g'_{a,b}(y_2^-)>1
\leqno (A.6)$$}

\begin{proof}
Firstly, we give an estimate for $g'_{a,b}(y_2^-)$.
We have
$$
g_{a,b}(y_2^-)=y_2^-+\dfrac{a}{e^{y_2^-}+1}-ab=y_{max}.
\leqno (A.7)$$
On the other hand, $y_2^->y_{min}$, hence
$$0<\dfrac{a}{e^{y_2^-}+1}<\dfrac{a}{e^{y_{min}}+1}<2$$
for a sufficiently big $a$.

Therefore, using (A.7) and the last inequality, we get $e^{y_2^-}>e^{y_{max}} e^{ab-2}$,
$$\dfrac{ae^{y_2^-}}{(e^{y_2^-}+1)^2}<ae^{-y_2^-}<ae^{-y_{max}} e^{2-ab}=ae^{y_{min}} e^{2-ab}
<a^2e^{-ab+3}$$
and
$$
 g'_{a,b}(y_2^-)>1-a^2 e^{-ab+3}
\leqno (A.8)$$
for big values of $a$.

Secondly, we have $g_{a,b}(y_1^+)=y_{max}$ that is
$$y_1^++\dfrac{a}{e^{y_1^+}+1}-ab=y_{max}$$
which means that
$$\dfrac{a}{e^{y_1^+}+1}=y_{max}-y_1^++ab$$
and, since $y_1^+\in (y_{max},y_{min})$, we have
$$\dfrac{a}{e^{y_1^+}+1}=ab+o(a)$$
and, consequently
$$\dfrac{1}{e^{y_1^+}+1}=b+o(1), \qquad e^{y_1^+}=\dfrac1{b}-1+o(1).$$

Therefore,
$$
 g'_{a,b}(y_1^+)=1-ab(1-b)+o(a)
\leqno (A.9)$$
which tends to $-\infty$ as $a\to \infty$.

Finally, we have $g_{a,b}(y_1^-)=y_2^+$ and $g_{a,b}(y_2^+)=y_{min}$. The last inequality together with the definition of $g_{a,b}$ and the fact that $y_2^+>y_{min}$ imply that $y_2^+=ab+o(a)$. This means that
$$y_1^-+\dfrac{a}{e^{y_1^-}+1}-ab =ab+o(a).$$
Since $y_1^-\in (y_{max},y_{min})$, we have
$$\dfrac{a}{e^{y_1^-}+1}=2ab+o(a),$$
$$\dfrac{1}{e^{y_1^-}+1}=2b+o(1), \qquad e^{y_1^-}=\dfrac1{2b}-1+o(1).$$

This means that
$$
 g'_{a,b}(y_1^{-})=1-2ab(1-2b)+o(a)
\leqno (A.10)$$
which also tends to $-\infty$ for big values of $a$.

Therefore, (A.6) directly follows from inequalities (A.8), (A.9), and (A.10). Thus, we complete the proof of Lemma A.4.
\end{proof}

Lemma A.4, together with the properties of the second derivative of $g_{a,b}$ (see (A.1)) and the fact that for any $x\in K$
$$ (x\in J_2)\implies (g_{a,b}(x) \in J_1)$$
implies that
$$
(g^2_{a,b})'(x)>1 \qquad \forall x\in K.
\eqno (A.11)$$
Thus the conjugacy between $g_{a,b}|_K$ and the shift of finite type on $\Sigma_0$ follows from (A.11) and the disjointedness of $J_1$ and $J_{2}$. Thus, Key Lemma A.1 is proven.
\end{proof}

\section*{Appendix B. Some estimates of the number of periodic solutions}
\setcounter{section}{8}

Finally, we estimate the number of periodic sequences in $\Sigma_0$.
Before formulating the result, we briefly refer to the history of the problem. Sharkovskii \cite{shark64} and, later on, Li and Yorke
(for a particular case) \cite{liyorke}, established the famous order on the set of natural numbers, so that a point of a 'higher' period implies the existence of a point of any 'lower' period. A bit later, Elaydi \cite{elaydi96} proved the converse result, establishing that for any $m\in {\mathbb N}$ there is a mapping with $m$ being the 'biggest' existing period according to the Sharkovskii's order.  What is especially important for us, is the result by Bau-Sen Du \cite{bausendu} (see also, \cite{ivanov} for a shorter proof). The general lower estimate on the value of periodic solution was given: in particular, the existence of a period-3 orbit for a 1D mapping $f$ implies that for any $n\in {\mathbb N}$ there exist at least $L_n$ solutions of the equation $f^n(x)=x$. Here $L_n$ is the Lucas number, related to Fibonacci's numbers as follows: $L_n=F_{n+1}+F_{n-1}$. 

\textbf{Proposition B.1.} \emph{Let $S$ be the natural shift mapping of the above set. Then for every $n\in {\mathbb N}$ the equation 
$$
S^n(x)=x
\eqno (B.1)$$
has at least $L_n$ distinct solutions.}

\begin{proof}
Let $B_n$ be the desired number. First of all, calculate the number $A_n$ of $n$-tuples of zeros and ones without repeating number 1 (we call them admissible). For $n=1$, we have $A_1=2$, for $n=2$, we have $A_3=3$. In addition, the number 0 can be concatenated to the right of any $n-1$ - tuple, while it is only sequences ending with $0$ that can be extended by $1$. So $A_{n+2}=A_{n}+A_{n+1}$ for any $n\in {\mathbb N}$ and, hence, $A_n=F_{n+2}$. To calculate $B_n$, we observe that any solution of (B.1) belongs to exactly one of two types:
\begin{enumerate}
\item it is a concatenation of an admissible $n-1$ - tuple and $0$;
\item it is a concatenation of an admissible $n-1$ - tuple, starting and ending with zeros and the number $1$. 
\end{enumerate}
Therefore, $B_{n+1}=A_{n}+A_{n-2}$ for any $n\ge 2$. Respecting the fact that $B_1=1$, $B_2=3$, we get the statement of the proposition.
\end{proof}

\begin{remark} This statement gives the estimate of the number of periodic solutions that exist due to Theorem 6.1. Although this estimate is similar to that given by Bau-Sen Du in the quoted paper, the latter proposition estimates the number of periodic points of a hyperbolic set $K$ from Lemma A.1. Other solutions, for example, stable ones, are not counted here. 
\end{remark}

\section*{Acknowledgments}
S. K. would like to thank Peking University for their support and hospitality. Y. Z. would like to thank the Mathematical Institute of the Polish Academy of Sciences for their support and hospitality, and is partially supported by NSFC Nos. 12161141002, 12271432 and USTC-AUST Math Basic Discipline Research Center. The authors would also like to thank Dr.\, Magdalena Chmara from Gda\'nsk University of Technology for her kind help. They are grateful to the organizers of the VIII Symposium on Nonlinear Analysis in Toru\'{n} 2024, where some parts of this work were done.

\end{document}